\documentclass[12pt,reqno]{amsart}
\usepackage{fullpage}
\usepackage{amsfonts}
\usepackage{amssymb}
\usepackage{times}
\usepackage{graphicx}
\newtheorem{thm}{Theorem}[section]

\newtheorem{lem}[thm]{Lemma}

\newtheorem{Definition}[thm]{Definition}

\newtheorem*{remark}{Remark}

\begin{document}

\title[Short version of title]{On the coefficient of The $n^{th}$ Cesaro mean of order $\alpha$
of bi- univalent functions}
% or if you want, simply \title{Title of the article}
\author{Adnan Ghazy Alamoush }

\maketitle

\begin{abstract}
The purpose of the present paper is to introduce a new subclasses of the function class  of bi-univalent functions defined in the open unit disc. Furthermore, we obtain estimates on the coefficients $|a_{2}|$
and $|a_{3}|$ for functions of this class. Some results related to this work will be briefly indicated.

\end{abstract}

\section{Introduction}
Let $A$ denote the class of the functions $f$ of the form
\begin{equation}\label{eg:}
    f(z)=z+\sum^{\infty}_{n=2} a_{n}z^{n}
\end{equation}
which are analytic in the open unit disc $ U=\{z\in C:|z|<1$\} and satisfy the normalization
condition $f(0)=f^{'}(0)-1=0$. Let $S$ be the subclass of A consisting of functions of the form (1) which are also univalent in $U$.

 A function $f\in A$  is said to be in the class of strongly bi-starlike functions of order $\alpha (0<\alpha\leq1)$,
denoted by $S^{*}_{\Sigma}(\alpha)$, if each of the following conditions is satisfied:

\begin{center}
$\left|\arg\left\{\frac{zf^{'}(z)}{f(z)}\right\}\right|<\frac{\alpha\pi}{2}, (|z|<1, 0<\alpha\leq1),$
\end{center}
and
\begin{center}
$\left|\arg\left\{\frac{zg^{'}(w)}{g(w)}\right\}\right|<\frac{\alpha\pi}{2}, (|w|<1, 0<\alpha\leq1),$
\end{center}
where $g$ is the extension of $f^{-1}$ to $U$ (for details see ( Brannan and Taha \cite{Brannan}).
And is said to be in the class of strongly bi-convex functions of order $\alpha(0<\alpha\leq1)$, denoted by $K_{\Sigma}(\alpha)$, if it
satisfies the following inequality
\begin{center}
$\left|\arg\left\{{1+\frac{zf^{''}(z)}{f^{'}(z)}}\right\}\right|<\frac{\alpha\pi}{2}, (|z|<1, 0<\alpha\leq1)$.
\end{center}
and

\begin{center}
$\left|\arg\left\{{1+\frac{wg^{''}(w)}{g^{'}(w)}}\right\}\right|<\frac{\alpha\pi}{2}, (|w|<1, 0<\alpha\leq1)$.
\end{center}
Where $g$ is the extension of $f^{-1}$ to $U$ .
Recall that the Koebe one-quarter theorem \cite{Duren2} ensures that the image of $D$ under every
univalent function $f\in S$ contains a disk of radius $\frac{1}{4}$. Thus every univalent function $f$
has an inverse $f^{-1}$ satisfying $f^{-1}(f(z)) = z$, $(z\in D)$ and
\begin{center}
$f^{-1}(f(w)) = w$, $(|w|<r_{0}f, r_{0}f\geq\frac{1}{4})$.
\end{center}
\begin{equation}\label{eg:}
         g(w)=w-a_{2}w^{2}+(2a_{2}^{2}-a_{3})w^{3}-(5a_{2}^{2}-5a_{2}a_{3}+a_{4})w^{4}+...\ .
\end{equation}
In recent years, many authors  discussed estimate on the
coefficients $|a_{2}|$ and  $|a_{3}|$ for subclasses of bi-univalent function (see for example \cite{Srivastava}, \cite{Frasin}, \cite{Alamoush1}, \cite{Alamoush2}, \cite{Alamoush3}, \cite{Alamoush4}).

Let $f:D\rightarrow C$ be an analytic function on $D$ having taylor expansion $f(z)=\Sigma^{\infty}_{n=1} a_{n}z^{n}, z\in D,$ with $a_{n}\in C$,\ $a_{1}=1 $ $,n=1,2,3,...\ .$ A function $f\in S$ is bi-univalent in $D$ if both $f$ and $f^{-1}$ are univalent in $D$.

\begin{flushleft}
The object of the present paper is to introduce a new subclasses of the function
class $\Sigma$ and to find estimates on the coefficients $|a_{2}|$ and  $|a_{3}|$ for new functions in these new subclasses of the function
class $\Sigma$.
 \begin{flushleft}

 We say that $\delta_{n}^{\alpha}f(z)$ is  The $n^{th}$ Cesaro mean of order $\alpha\geq0$ of $f$ is defined by
 \end{flushleft}

\ \\
\end{flushleft}
\begin{center}
$\delta_{n}^{\alpha}f(z)=z+\sum^{\infty}_{n=2}A_{n} a_{n}z^{n}$
\end{center}
where
 \begin{center}
$ A_{n}=   {\begin{array}{c}\frac { \left ( {\begin{array}{c} k+\alpha-n\\ k-n \end{array}} \right ) }
{ \left ( {\begin{array}{c} k+\alpha-1 \\ k-n\end{array}} \right ) } \end{array}},\ a_{1}=1.   $
\\

\end{center}
Let $D$ denote the open unit disk in $C$. It is well known that outer functions are zero-free on the unit disk. Outer functions, which play an important role in $H_{p}$ theory to find a suitable finite (polynomial) approximation for the outer infinite series $f$ so that the approximant reduces the zero-free property of $f$, arise in the characteristic equation which determines the stability of certain nonlinear systems of differential equations. Recall that an outer function is a function $f \in H_{p}$ of the form
\begin{center}

$f(z) = e^{i \gamma} e^{ \frac{1} {2 \pi} \int _{-\pi} ^ {\pi}
\frac{1 + e^{it}z} {1-e^{it}z} \log \psi (t)dt }$
\end{center}

\noindent where  $\psi(t) \geq 0$,    $\log  \psi(t)$   is in $L^1$ and  $\psi(t)$ is in $L^p$.  See \cite{Duren1} for the definitions and classical properties of outer functions.  Since any function $f$ in $H^1$ which has $1/f$ in $H^1$ is an outer function, then typical examples of outer functions can be generated by functions of the form
$\prod _ {k=1} ^ n (1 - e^{i \theta _ k}z)^{\alpha _ k}$ for $-1 <  \alpha _k < 1$.

We observe for outer functions that the standard Taylor approximants do not, in general, retain
the zero-free property of $f$. It was shown in \cite{Barnard} that the Taylor approximating polynomials to outer functions can vanish in the unit disk. By using convolution methods that the classical Cesaro means, retains the zero-free property of the derivatives of bounded convex functions in the unit disk. The classical Cesaro means play an important role in geometric function theory (see \cite{Ruscheweyh1},\cite{Ruscheweyh2}).

\begin{lem}

\textit{If $h\in p$  then $|c_{k}|<1,$  for each $k$, where $p$ is the family of all functions $h$ analytic in $U$ for which $\Re\{h(z)\}>0$, then}
\begin{center}
$h(z)=1+c_{1}z+c_{2}z^{2}+c_{3}z^{3}+...\ , z\in U$.
\end{center}
\end{lem}
\section{COEFFICIENT BOUNDS FOR THE FUNCTION CLASS $H_{\Sigma}(\psi)$}

In the sequel, it is assumed that $\varphi$ . is an analytic function with positive real part in the
unit disk $D$, satisfying $\psi(0) = 1, \psi^{'}(0) > 0$, and $\psi(D)$ is symmetric with respect to the real
axis. Such a function has a Taylor series of the form
\begin{center}
\begin{equation}\label{eg:}
    \psi(z)=1+B_{1}z+B_{2}z^{2}+B_{3}z^{3}+...., (B_{1}>0).
\end{equation}
\end{center}
Suppose that $u(z)$ and $v(z)$ are analytic in the unit disk D with $u(0) = v(0) = 0$,
$|u(z)| < 1,|v(z)| < 1$, and suppose that

\begin{center}
\begin{equation}\label{eg:}
u(z)=b_{1}z+\sum^{\infty}_{n=2} b_{n}z^{n}, v(z)=c_{1}z+\sum^{\infty}_{n=2} c_{n}z^{n},(|z|<1).
\end{equation}
\end{center}

It is well known that

\begin{center}
\begin{equation}\label{eg:}
|b_{1}|\leq1, |b_{2}|\leq1-|b_{1}|^{2},  |c_{1}|\leq1 , |c_{2}|\leq1-|c_{1}|^{2}.
\end{equation}
\end{center}

By a simple calculation, we have

\begin{center}
\begin{equation}\label{eg:}
\psi(u(z))=1+B_{1}b_{1}z+(B_{1}b_{2}+B_{2}b^{2}_{1})z^{2}+...\ , |z|<1
\end{equation}
\end{center}
and

\begin{center}
\begin{equation}\label{eg:}
\psi(v(w))=1+B_{1}c_{1}w+(B_{1}c_{2}+B_{2}c^{2}_{1})w^{2}+...\ , |w|<1.
\end{equation}
\end{center}

\begin{Definition}
\cite{Ali} A function $f\in \Sigma$ is said to be in the class $H_{\Sigma}(\psi)$ if and only if
\begin{center}
$f^{'}(z) \prec \psi(z),g^{'}(z) \prec \psi(w)$,
\end{center}
where $g(w) = f^{-1}(w)$.
\end{Definition}

\begin{thm}
\label{eg:}
If $f$ given by (1) is in the class $H_{\Sigma}(k,\psi)$, then

\begin{center}
\begin{equation}\label{eg:}
|a_{2}|\leq\left|\begin{array}{c}\frac { \left ( {\begin{array}{c} k+\alpha-1\\ k-2 \end{array}} \right ) }
{ \left ( {\begin{array}{c} k+\alpha-2 \\ k-2\end{array}} \right ) } \end{array}\right|\frac{B_{1}\sqrt{B_{1}}}{\sqrt{[|3B^{2}_{1}-4B_{2}|+4B_{1}]}}
\end{equation}
\end{center}

\begin{flushleft}
and
\end{flushleft}

\begin{equation}\label{eg:}
    |a_{3}|\leq\left[\begin{array}{c}\frac { \left ( {\begin{array}{c} k+\alpha-1\\ k-3 \end{array}} \right ) }
{ \left ( {\begin{array}{c} k+\alpha-3 \\ k-3\end{array}} \right ) } \end{array}\right]\left[(1-\frac{4}{3B_{1}})\frac{B^{3}_{1}}{{[|3B^{2}_{1}-4B_{2}|+4B_{1}]}}+\frac{B_{1}}{3}\right].
\end{equation}
\end{thm}

\begin{proof}
Let $f\in H_{\Sigma}(k,\psi)$  and $g = f^{-1}$. Where $a_{1}=1$ . Then there are analytic functions $u, v : D \rightarrow D$
given by (4) such that
\begin{equation}\label{eg:}
    [\delta_{n}^{\alpha}f(z)]^{'}=\psi(u(z)),  [\delta_{n}^{\alpha}g(w)]^{'}=\psi(v(w)),
\end{equation}

since

\begin{center}
\begin{equation}\label{eg:}
    [\delta_{n}^{\alpha}f(z)]^{'}=1+2 A_{2}a_{2}z+3 A_{3}a_{3}z^{3}+...\ ,
\end{equation}
\end{center}
\begin{center}
$[\delta_{n}^{\alpha}g(w)]^{'}=1-2A_{2}a_{2}w+3[2 A_{2}^{2}a^{2}_{2}-A_{3}a_{3}]w^{3}+...\ ,$
\end{center}
it follows from (6), (7), (10) and (11) that
\begin{center}
\begin{equation}\label{eg:}
    2A_{2}a_{2}=B_{1}b_{1},
\end{equation}
\begin{equation}\label{eg:}
    3A_{3}a_{3}=B_{1}b_{2}+B_{2}b^{2}_{1},
\end{equation}
\end{center}

\begin{center}
\begin{equation}\label{eg:}
   - 2A_{2}a_{2}=B_{1}c_{1},
\end{equation}
\end{center}

\begin{center}
\begin{equation}\label{eg:}
3[2A^{2}_{2}a^{2}_{2}-A_{3}a_{3}]=B_{1}c_{2}+B_{2}c^{2}_{1}.
\end{equation}
\end{center}
From (12) and (14), we get
\begin{center}
\begin{equation}\label{eg:}
b_{1}=-c_{1}.
\end{equation}
\end{center}
By adding (15) to (13), further computations using (12) and (16) lead to

\begin{center}
\begin{equation}\label{eg:}
A_{2}^{2}a_{2}^{2}[3B^{2}_{1}-8B_{2}]=B_{1}^{3}(b_{2}+c_{2}).
\end{equation}
\end{center}
Also, from (16) and  (17), together with (5), we obtain

\begin{center}
\begin{equation}\label{eg:}
|A_{2}^{2}a_{2}^{2}[3B^{2}_{1}-8B_{2}]|\leq2 B_{1}^{3}(1-|b_{1}|^{2}).
\end{equation}
\end{center}
From (12) and (18) we get
\begin{center}
$|a_{2}|\leq\left|\begin{array}{c}\frac { \left ( {\begin{array}{c} k+\alpha-1\\ k-2 \end{array}} \right ) }
{ \left ( {\begin{array}{c} k+\alpha-2 \\ k-2\end{array}} \right ) } \end{array}\right|\frac{B_{1}\sqrt{B_{1}}}{\sqrt{[|3B^{2}_{1}-4B_{2}|+4B_{1}]}}$.
\end{center}

Which, in view of the well-known inequalities $|b_{2}|\leq 2$ and $|c_{2}|\leq 2$ for functions with positive real part, gives us the desired estimate on $|a_{2}|$ as asserted in (8).
By subtracting (15) from (13), further computations using (12) and (16) lead to
\begin{center}
\begin{equation}\label{eg:}
6A_{3}a_{3}=6A_{2}^{2}a^{2}_{2}+B_{1}(b_{2}-c_{2}).
\end{equation}
\end{center}
From (5), (12), (16) and (19), it follows that

\begin{center}
$|a_{3}|=\frac{6A_{2}^{2}|a_{2}|^{2}+B_{1}(|b_{2}|+|c_{2}|)}{6A_{3}}$
\end{center}
\begin{center}
$\leq\frac{6A_{2}^{2}|a_{2}|^{2}+B_{1}(1-|b_{1}|^{2})+(1-|c_{1}|^{2})}{6A_{3}}$
\end{center}

\begin{center}
$\leq\frac{[1-\frac{4}{3B_{1}}]A_{2}^{2}|a_{2}|^{2}}{A_{3}}+\frac{B_{1}}{3A_{3}}$
\end{center}

\begin{center}
$|a_{3}|\leq\left[\begin{array}{c}\frac { \left ( {\begin{array}{c} k+\alpha-1\\ k-3 \end{array}} \right ) }
{ \left ( {\begin{array}{c} k+\alpha-3 \\ k-3\end{array}} \right ) } \end{array}\right]\left[(1-\frac{4}{3B_{1}})\frac{B^{3}_{1}}{{[|3B^{2}_{1}-4B_{2}|+4B_{1}]}}+\frac{B_{1}}{3}\right]$.
\end{center}
\end{proof}

\section{COEFFICIENT BOUNDS FOR THE FUNCTION CLASS $Q_{\Sigma}(\alpha,\mu,\lambda)$ }

\begin{Definition}
\textit{A function $f(z)$ given by (1) is said to be in the class $Q_{\Sigma}(\alpha,\mu,\lambda)$   if the following conditions are satisfied}: For $f\in \Sigma,$
\\
\begin{equation}\label{dlabel}
 \left|arg\left\{\frac{(1-\lambda)\delta^{\alpha}_{_{n}}f(z)+\lambda z[\delta^{\alpha}_{_{n}}f(z)]'}z\right\}\right|<\frac{\pi\alpha}{2}, \alpha(0 <\alpha \leq 1, \lambda\geq1, z\in U),
\end{equation}
\begin{flushleft}
and
\end{flushleft}
\begin{equation}\label{eg:}
\left|arg\left\{\frac{(1-\lambda)\delta^{\alpha}_{_{n}}g(w)+\lambda w[\delta^{\alpha}_{_{n}}g(w)]'}w\right\}\right|<\frac{\pi\alpha}{2},\alpha(0 <\alpha \leq 1,\lambda\geq1, w\in U),
\end{equation}
\begin{flushleft}
where the function $g$ defined by (2).
\end{flushleft}
\end{Definition}

   \begin{thm}
\textit{
Let the function $f(z)$ given by (1) be in the class $Q_{\Sigma}(\alpha,\mu,\lambda)$,
$n\in N_{0},0\leq\beta<1, \lambda\geq1$. Then}
\begin{equation}\label{eg:}
|a_{2}|\leq2\alpha\left|\begin{array}{c}\frac { \left ( {\begin{array}{c} k+\alpha-1\\ k-2 \end{array}} \right ) }
{ \left ( {\begin{array}{c} k+\alpha-2 \\ k-2\end{array}} \right ) } \end{array}\right|\frac{{1}}{\sqrt{4^{k}(1+\lambda)^{2}+\alpha[2.3^{k}(1+\lambda)-4^{k}(1+\lambda)^{2}]}},
\end{equation}
and
\begin{equation}\label{eg:}
|a_{3}|\leq\left[\begin{array}{c}\frac { \left ( {\begin{array}{c} k+\alpha-1\\ k-3 \end{array}} \right ) }
{ \left ( {\begin{array}{c} k+\alpha-3 \\ k-3\end{array}} \right ) } \end{array}\right]\left[\frac{2\alpha}{(1+2\lambda)}+\frac{4\alpha^{2}}{(1+\lambda)^{2}}\right].\end{equation}

\end{thm}
\begin{proof}
From (20) and (21), we can write
\begin{equation}\label{eq:}
\frac{(1-\lambda)\delta^{\alpha}_{_{n}}f(z)+\lambda z[\delta^{\alpha}_{_{n}}f(z)]'}z=[p(z)]^{\alpha},
\end{equation}
and

\begin{equation}\label{eq:}
\frac{(1-\lambda)\delta^{\alpha}_{_{n}}g(w)+\lambda w[\delta^{\alpha}_{_{n}}g(w)]'}w=[q(w)]^{\alpha},
\end{equation}
where $p(z)$ and $q(w)$ in $P$ and  $a_{1}=1$, and have the forms

\begin{equation}\label{eq:}
p(z)=1+p_{1}z+p_{2}z^{2}+p_{3}z^{3}+...\ \ ,
\end{equation}
and
\begin{equation}\label{eq:}
q(w)=1+p_{1}w+q_{2}w^{2}+q_{3}w^{3}+...\ \ .
\end{equation}
Now, equating the coefficients in (24) and (25), we obtain
\begin{equation}\label{eg:}
    (1+\lambda)A_{2}a_{2}=\alpha p_{1},
\end{equation}
\begin{equation}\label{eg:}
        (1+2\lambda)A_{3}a_{3}=\frac{1}{2}[2\alpha p_{2}+\alpha(\alpha-1)p_{1}^{2}],
\end{equation}
\begin{equation}\label{eg:}
    - (1+\lambda)A_{2}a_{2}=\alpha q_{1},
\end{equation}
\begin{equation}\label{eg:}
   (1+2\lambda)[2A_{2}^{2}a^{2}_{2}-A_{3}a_{3}]=\frac{1}{2}[2\alpha q_{2}+\alpha(\alpha-1)q_{1}^{2}].
\end{equation}
From (28) and (30), we obtain
\begin{equation}\label{eg:}
p_{1}=-q_{1}
\end{equation}
and
\begin{equation}\label{eg:}
         2(1+\lambda)^{2}A_{2}^{2}a^{2}_{2}=\alpha ^{2}(p^{2}_{1}+q^{2}_{1}).
\end{equation}
Now, from (29), (31) and (33), we obtain

\begin{center}
       $  2(1+2\lambda)A_{2}^{2}a^{2}_{2}=\alpha (p_{2}+q_{2})+\frac{1}{2}[2\alpha(\alpha-1)(p_{1}^{2}+q_{1}^{2})]$
\end{center}

\begin{center}
$=\alpha(p_{2}+q_{2})+\frac{\alpha(\alpha-1)}{2}.\frac{2(1+\lambda)^{2}A_{2}^{2}a^{2}_{2}}{\alpha^{2}}\ .$
\end{center}
Therefore we have

    \begin{center}
    $a^{2}_{2}=\frac{\alpha^{2}(p_{2}+q_{2})}{[(1+\lambda)^{2}+\alpha[(1+2\lambda)]-\lambda^{2}]]A_{2}^{2}}\ .$
    \end{center}

 Applying Lemma 1.1 for the coefficients $p_{2}$ and $q_{2}$, we immediately have
\begin{center}
 $|a_{2}|\leq2\alpha\left|\begin{array}{c}\frac { \left ( {\begin{array}{c} k+\alpha-1\\ k-2 \end{array}} \right ) }
{ \left ( {\begin{array}{c} k+\alpha-2 \\ k-2\end{array}} \right ) } \end{array}\right|\frac{{1}}{\sqrt{4^{k}(1+\lambda)^{2}+\alpha[2.3^{k}(1+\lambda)-4^{k}(1+\lambda)^{2}]}}$.
\end{center}
This gives the bound as asserted in (22).

Next, in order to find the bound on $|a_{3}|$, we subtract (29) from (31) and obtain
\begin{center}
\begin{quote}
\begin{center}
$2[(1+2\lambda)(A_{3}a_{3}-A_{2}^{2}a^{2}_{2})$
\end{center}
\end{quote}
=$\frac{1}{2}(2\alpha(p_{2}-q_{2})+\alpha(\alpha-1)(p^{2}_{1}-q^{2}_{1})),$
\end{center}

\begin{center}
$a_{3}=\frac{\alpha(p_{2}-q_{2})}{2(1+2\lambda)A_{3}}+\frac{\alpha^{2}(p^{2}_{1}+q^{2}_{1})}{2(1+\lambda)^{2}A_{3}}$,
\end{center}

\begin{center}
$a_{3}=\left[\begin{array}{c}\frac { \left ( {\begin{array}{c} k+\alpha-1\\ k-3 \end{array}} \right ) }
{ \left ( {\begin{array}{c} k+\alpha-3 \\ k-3\end{array}} \right ) } \end{array}\right]\left[\frac{\alpha(p_{2}-q_{2})}{2(1+2\lambda)}+\frac{\alpha^{2}(p^{2}_{1}+q^{2}_{1})}{2(1+\lambda)^{2}}\right]$,
\end{center}
 Applying Lemma 1.1 for the coefficients $p_{2}$ and $q_{2}$, we immediately have
\begin{center}
$|a_{3}|\leq\left[\begin{array}{c}\frac { \left ( {\begin{array}{c} k+\alpha-1\\ k-3 \end{array}} \right ) }
{ \left ( {\begin{array}{c} k+\alpha-3 \\ k-3\end{array}} \right ) } \end{array}\right]\left[\frac{2\alpha}{(1+2\lambda)}+\frac{4\alpha^{2}}{(1+\lambda)^{2}}\right]$.
\end{center}

This completes the proof of Theorem 3.2.
\end{proof}

\section{COEFFICIENT BOUNDS FOR THE FUNCTION CLASS   $H_{\Sigma}(\beta,\mu,\lambda)$ }

\begin{Definition}
\textit{
A function $f(z)$ given by (1) is said to be in the class $H_{\Sigma}(\beta,\mu,\lambda)$  if the following conditions are satisfied: For $f\in \Sigma,$}
\begin{equation}\label{eq:}
 \Re\left\{\frac{(1-\lambda)\delta^{\alpha}_{_{n}}f(z)+\lambda z[\delta^{\alpha}_{_{n}}f(z)]'}z\right\}>\beta,z\in U,n\in N_{0},0\leq\beta<1, \lambda\geq1.
\end{equation}

and

\begin{equation}\label{eq:}
 \Re\left\{\frac{(1-\lambda)\delta^{\alpha}_{_{n}}g(w)+\lambda w [\delta^{\alpha}_{_{n}}g(w)]'}w\right\}>\beta,w\in U,n\in N_{0},0\leq\beta<1, \lambda\geq1,
\end{equation}
\end{Definition}
where the function $g(z)$ defined by (2).

\ \\

\begin{thm}.
Let $f (z)$ given by (1) be in the class $H_{\Sigma}(\beta,\mu,\lambda), 0\leq\beta<1, \mu\geq0,$ and $\lambda\geq1.$ Then
\begin{equation}\label{eg:}
|a_{2}|\leq\left[\begin{array}{c}\frac { \left ( {\begin{array}{c} k+\alpha-1\\ k-2 \end{array}} \right ) }
{ \left ( {\begin{array}{c} k+\alpha-2 \\ k-2\end{array}} \right ) } \end{array}\right]\sqrt{\frac{2(1-\beta)}{1+2\lambda}}\end{equation}
\begin{flushleft}
and
\end{flushleft}

\begin{equation}\label{eg:}
    |a_{3}|\leq\left[\begin{array}{c}\frac { \left ( {\begin{array}{c} k+\alpha-1\\ k-3 \end{array}} \right ) }
{ \left ( {\begin{array}{c} k+\alpha-3 \\ k-3\end{array}} \right ) } \end{array}\right]\left[\frac{4(1-\beta)^{2}}{(1+\lambda)^{2}}+
    \frac{2(1-\beta)}{(1+2\lambda)}\right].
\end{equation}
\end{thm}

\begin{proof}.
It follows from (34) and (35) that there exists $p,q\in P$ such that
\begin{equation}\label{dlabel}\frac{{(1-\lambda)\delta^{\alpha}_{_{n}}f(z)+\lambda z [\delta^{\alpha}_{_{n}}f(z)]^{'}}}{z}=\beta+(1-\beta)p(z),\end{equation}

and
\begin{equation}\label{dlabel}\frac{{(1-\lambda)\delta^{\alpha}_{_{n}}g(w)+\lambda w [\delta^{\alpha}_{_{n}}g(w)}]^{'}}{w}=\beta+(1-\beta)q(w),\end{equation}

\begin{flushleft}
where $a_{1}=1$, and have the forms

\begin{equation}\label{eq:}
p(z)=1+p_{1}z+p_{2}z^{2}+p_{3}z^{3}+.....\ \ ,
\end{equation}
and
\begin{equation}\label{eq:}
q(w)=1+p_{1}w+q_{2}w^{2}+q_{3}w^{3}+.....\ \ ,
\end{equation}
 respectively. Equating coefficients in (38) and (39) yields
\end{flushleft}
\begin{equation}\label{eg:}
    [(1+\lambda)A_{2}a_{2}=(1-\beta) p_{1},
\end{equation}
\begin{equation}\label{eg:}
        [(1+2\lambda)A_{3}a_{3}=(1-\beta)p_{2},
\end{equation}
\begin{equation}\label{eg:}
-    [(1+\lambda)A_{2}a_{2}=(1-\beta) q_{1},
\end{equation}
and
\begin{equation}\label{eg:}
        (1+2\lambda)[2A_{2}^{2}a^{2}_{2}-A_{3}a_{3}]=(1-\beta)q_{2}.
\end{equation}

From (42) and (44), we have

\begin{equation}\label{eg:}
-p_{1}=q_{1}\end{equation}
and
\begin{equation}\label{eg:}
            2(1+\lambda)^{2}A_{2}^{2}a^{2}_{2}=(1-\beta)^{2}(p^{2}_{1}+ q^{2}_{1}).
\end{equation}
Also, from (43) and (45), we find that
\begin{equation}\label{eg:}
            2(1+2\lambda)A_{2}^{2}a^{2}_{2}=(1-\beta)(p_{2}+ q_{2}),
\end{equation}
\begin{equation}\label{eg:}
|a^{2}_{2}|\leq\left[\begin{array}{c}\frac { \left ( {\begin{array}{c} k+\alpha-1\\ k-2 \end{array}} \right ) }
{ \left ( {\begin{array}{c} k+\alpha-2 \\ k-2\end{array}} \right ) } \end{array}\right]\frac{(1-\beta)(|p_{2}|+|q_{2}|)}{2(1+2\lambda)},\end{equation}

 \begin{flushleft}
 $\Rightarrow$
 \end{flushleft}
\begin{equation}\label{eg:}
|a_{2}|\leq\left[\begin{array}{c}\frac { \left ( {\begin{array}{c} k+\alpha-1\\ k-2 \end{array}} \right ) }
{ \left ( {\begin{array}{c} k+\alpha-2 \\ k-2\end{array}} \right ) } \end{array}\right]\sqrt{\frac{2(1-\beta)}{1+2\lambda}},\end{equation}
which is the bound on $|a_{2}|$ as given in (36).

Next, in order to find the bound on $|a_{3}|$ by subtracting
(45) from (43), we obtain

\begin{center}
    $ 2A_{3}(1+2\lambda)a_{3}=$
    \end{center}
\begin{center}
    $2(1+2\lambda)A_{2}^{2}a^{2}_{2}
    +(1-\beta)(p_{2}-q_{2})$

\end{center}

or, equivalently

\begin{center}
    $ a_{3}=\frac{2(1+2\lambda)A_{2}^{2}a^{2}_{2}}{2A_{3}(1+2\lambda)}$
  $ +\frac{(1-\beta)(p_{2}-q_{2})}{2A_{3}(1+2\lambda)}.$
    \end{center}

Upon substituting the value of $a_{2}^{2}$
from (47), we obtain

\begin{center}
   $ a_{3}=A_{3}\left[\frac{(1-\beta)^{2}(p^{2}_{1}+q^{2}_{1})}{2(1+\lambda)^{2}}+
    \frac{(1-\beta)(p_{2}-q_{2})}{2(1+2\lambda)}\right]$.
\end{center}

Applying Lemma 1.1 for the coefficients $p_{1},p_{2},q_{1}$  and $q_{2}$  we obtain
\begin{center}

   $ a_{3}=A_{3}\left[\frac{4(1-\beta)^{2}}{(1+\lambda)^{2}}+
    \frac{2(1-\beta)(p_{2}-q_{2})}{(1+2\lambda)}\right].$
\end{center}
\begin{center}

    $|a_{3}|\leq\left[\begin{array}{c}\frac { \left ( {\begin{array}{c} k+\alpha-1\\ k-3 \end{array}} \right ) }
{ \left ( {\begin{array}{c} k+\alpha-3 \\ k-3\end{array}} \right ) } \end{array}\right]\left[\frac{4(1-\beta)^{2}}{(1+\lambda)^{2}}+
    \frac{2(1-\beta)}{(1+2\lambda)}\right].$
\end{center}

which is the bound on $|a_{3}|$ as asserted in (37).
\end{proof}
\  \\
\begin{remark}1.
\textit{  For all $\alpha\geq0$, and $k=n$ in Theorems 2.2, we obtain the corresponding results due to Zhigang and Qiuqiu \cite{Zhigang}
}\end{remark}

\begin{remark}2.
\textit{  For all $\alpha\geq0$, and $k=n$ in Theorems 3.2 and 4.2, we obtain the corresponding results due to Frasin and Aouf \cite{Frasin}.
}\end{remark}

\ \\

   Adnan Ghazy Alamoush

\sl \small Faculty of Science, Taibah University, Saudi Aarabia.\\

Email:\textit{ adnan--omoush@yahoo.com}\\
%
%-------------------------------------------------------------------
%
\end{document}